\documentclass[12pt]{article}
\usepackage{amsfonts}
\usepackage{amsmath}
\usepackage{amsthm}
\numberwithin{equation}{section}
\renewcommand{\O}{\mathcal{O}}
\newtheorem{theorem}{Theorem}[section]
\begin{document}

\title{Nearest neighbor recurrence relations for multiple orthogonal polynomials}
\author{Walter Van Assche \\
Katholieke Universiteit Leuven, Belgium}
\date{\today}
\maketitle
\begin{abstract}
We show that multiple orthogonal polynomials for $r$ measures $(\mu_1,\ldots,\mu_r)$ satisfy a system of linear recurrence relations only involving nearest neighbor multi-indices $\vec{n}\pm \vec{e}_j$, where $\vec{e}_j$ are the standard unit vectors. The recurrence coefficients are not arbitrary but satisfy a system of partial difference equations with boundary values given by the recurrence coefficients of the orthogonal polynomials with each of measures $\mu_j$. We show how the Christoffel-Darboux formula for multiple orthogonal polynomials can be obtained easily using this information.
We give explicit examples involving multiple Hermite, Charlier, Laguerre, and Jacobi polynomials. 
\end{abstract}

\section{Introduction}  
The three-term recurrence relation is an important piece of information when one is studying orthogonal polynomials.
If $\mu$ is a positive measure on the real line for which all the moments
\[   \nu_n = \int x^n \, d\mu(x)  \]
exist, then we can define orthogonal polynomials $\{P_n,\ n=0,1,2,\ldots\}$ by
\begin{equation} \label{eq:1.1}
         \int P_n(x) x^k\, d\mu(x) = 0, \qquad k=0,1,\ldots,n-1.
\end{equation}
if we normalize the polynomials to be monic, i.e., $P_n(x) = x^n + \cdots$, then these polynomials always exist
and they are unique whenever the support of $\mu$ is an infinite set, and this is due to the fact that
the Hankel matrix of moments
\begin{equation*}  \label{eq:1.2}
    M_n = \begin{pmatrix} \nu_0 & \nu_1 & \cdots & \nu_n \\
                          \nu_1 & \nu_2 & \cdots & \nu_{n+1} \\
                          \vdots & \vdots & \cdots & \vdots \\
                          \nu_n & \nu_{n+1} & \cdots & \nu_{2n} \end{pmatrix} 
\end{equation*}
is positive definite for all $n \geq 0$. The monic orthogonal polynomials are explicitly given by
\begin{equation*}  \label{eq:1.3}
      P_n(x) = \frac{1}{\det M_{n-1}} \begin{vmatrix}  
			  \nu_0 & \nu_1 & \cdots & \nu_n \\
                          \nu_1 & \nu_2 & \cdots & \nu_{n+1} \\
                          \vdots & \vdots & \cdots & \vdots \\
                          \nu_{n-1} & \nu_{n} & \cdots & \nu_{2n-1} \\
                            1 & x & \cdots & x^n 
					\end{vmatrix}.
\end{equation*} 
They satisfy a three-term recurrence relation
\begin{equation} \label{eq:1.4}
     P_{n+1}(x) = (x-b_n) P_n(x) - a_n^2 P_{n-1}(x),
\end{equation}
with initial values $P_0=1$ and $P_{-1}=0$ and recurrence coefficients $b_n \in \mathbb{R}$
and $a_n^2 >0$. An important result (the spectral theorem for orthogonal polynomials or Favard's theorem)
is that any choice of recurrence coefficients $b_n \in \mathbb{R}$ and $a_{n+1}^2 >0$ $(n=0,1,2,\ldots)$
gives a sequence of monic orthogonal polynomials for some positive measure $\mu$ on the real line.
This measure $\mu$ is the spectral measure of the Jacobi operator
\begin{equation*}  \label{eq:1.5}
       J = \begin{pmatrix}   b_0 & a_1 & 0 & 0 & 0 & 0 &\cdots \\[4pt]
                             a_1 & b_1 & a_2 & 0 & 0 & 0 & \cdots \\[4pt]
                              0 & a_2 & b_2 & a_3 & 0 & 0 & \cdots \\
                             0 & 0 & a_3 & \ddots & \ddots & &  \\
                             \vdots & \vdots &  & \ddots
           \end{pmatrix}  
\end{equation*}
or, when the operator is not self-adjoint, a self-adjoint extension of this operator.
For the general theory of orthogonal polynomials we refer to Chihara \cite{Chi}, Freud \cite{Fre}, 
Gautschi \cite{Gau}, Ismail \cite{Ismail} and Szeg\H{o} \cite{Sze},
and for the spectral theory we recommend Simon \cite{Sim}.

Recently a generalization of orthogonal polynomials has been considered. Let $\vec{n} = (n_1,n_2,\ldots,n_r) \in
\mathbb{N}^r$ be a multi-index of size $|\vec{n}| = n_1+n_2+\ldots+n_r$ and suppose $\mu_1,\mu_2,\ldots,\mu_r$
are positive measure on the real line. \textbf{Type I multiple orthogonals} are given by the vector $(A_{\vec{n},1},\ldots,
A_{\vec{n},r})$, where $A_{\vec{n},j}$ is a polynomial of degree $\leq n_j-1$, for which
\begin{equation}  \label{eq:1.6}
     \int  x^k \sum_{j=1}^r A_{\vec{n},j}(x) w_j(x)  \, d\mu(x) = 0, 
    \qquad k=0,1,\ldots,|\vec{n}|-2, 
\end{equation}
where $\mu = \mu_1+\mu_2+\cdots+\mu_r$ and $w_j$ is the Radon-Nikodym derivative $d\mu_j/d\mu$. We use the
normalization
\begin{equation} \label{eq:1.7}
       \int  x^{|\vec{n}|-1} \sum_{j=1}^r A_{\vec{n},j}(x) w_j(x) \, d\mu(x) = 1.
\end{equation}
The equations \eqref{eq:1.6}--\eqref{eq:1.7} are a linear system of $|\vec{n}|$ equations for the unknown
coefficients of $A_{\vec{n},1},\ldots,A_{\vec{n},r}$. This system has a unique solution if the matrix of mixed moments
\begin{equation} \label{eq:1.8}
    M_{\vec{n}} = \left( \begin{matrix} 
                   \nu_0^{(1)} & \nu_1^{(1)} & \cdots & \nu_{n_1-1}^{(1)} \\
                   \nu_1^{(1)} & \nu_2^{(1)} & \cdots & \nu_{n_1}^{(1)} \\
	            \vdots & \vdots & \cdots & \vdots \\
                   \nu_{|\vec{n}|-1}^{(1)} & \nu_{|\vec{n}|}^{(1)} & \cdots & \nu_{\vec{n}+n_1-2}^{(1)} 		
                  \end{matrix} \right|  
		\cdots \cdots
                 \left| \begin{matrix} 
                   \nu_0^{(r)} & \nu_1^{(r)} & \cdots & \nu_{n_r-1}^{(r)} \\
                   \nu_1^{(r)} & \nu_2^{(r)} & \cdots & \nu_{n_r}^{(r)} \\
	            \vdots & \vdots & \cdots & \vdots \\
                   \nu_{|\vec{n}|-1}^{(r)} & \nu_{|\vec{n}|}^{(r)} & \cdots & \nu_{\vec{n}+n_r-2}^{(r)} 		
                  \end{matrix} \right) ,
\end{equation}
where
\[    \nu_n^{(j)} = \int x^n\, d\mu_j(x), \qquad j=1,\ldots,r,  \]
is not singular, in which case we call the multi-index $\vec{n}$ a \textit{normal index}.
The \textbf{type II multiple orthogonal polynomial} is the monic polynomial $P_{\vec{n}}(x) = x^{|\vec{n}|} + \cdots$
of degree $|\vec{n}|$ for which
\begin{eqnarray} \label{eq:1.9}
    \int P_{\vec{n}}(x) x^k\, d\mu_1(x) &=& 0, \qquad k=0,1,\ldots,n_1-1, \nonumber \\
	\vdots	& & \\
    \int P_{\vec{n}}(x) x^k\, d\mu_r(x) &=& 0, \qquad k=0,1,\ldots,n_r-1. \nonumber
\end{eqnarray}
This gives a linear system of $|\vec{n}|$ equations for the $|\vec{n}|$ unknown coefficients of
$P_{\vec{n}}$ and the matrix of this lineéar system is the transpose $M_{\vec{n}}^t$ of \eqref{eq:1.8},
hence $P_{\vec{n}}$ exists and is unique whenever $\vec{n}$ is a normal index.
See Aptekarev \cite{Apt}, Coussement and Van Assche \cite{WVAEC}, Nikishin and Sorokin \cite[Chapter 4, \S 3]{NikSor},
Ismail \cite[Chapter 23]{Ismail}. 
The polynomials on the stepline, i.e.,
\[     p_{kr+j} = P_{(k+1,\ldots,k+1,k,\ldots,k)} \]
are also known as $d$-orthogonal polynomials (with $d=r$), see Douak and Maroni \cite{DouMar},
Ben Cheikh and Douak \cite{BChDou}. 

In this paper we assume that all multi-indices are normal and we investigate the nearest-neighbor recurrence relations
\begin{eqnarray} \label{eq:1.10}  
   x P_{\vec{n}}(x) &=& P_{\vec{n}+\vec{e}_1}(x) + b_{\vec{n},1}P_{\vec{n}}(x) + \sum_{j=1}^r a_{\vec{n},j} P_{\vec{n}-\vec{e}_j}(x), 
   \nonumber \\
      \vdots & &  \\
   x P_{\vec{n}}(x) & = & P_{\vec{n}+\vec{e}_r}(x) + b_{\vec{n},r}P_{\vec{n}}(x) + \sum_{j=1}^r a_{\vec{n},j} P_{\vec{n}-\vec{e}_j}(x),
   \nonumber
\end{eqnarray}
where $\vec{e}_j = (0,\ldots,0,1,0,\ldots,0)$ is the $j$-th standard unit vector with 1 on the $j$-th entry,
and $(a_{\vec{n},1},\ldots,a_{\vec{n},r})$ and $(b_{\vec{n},1},\ldots,b_{\vec{n},r})$ are the recurrence coefficients.
These recurrence relations were derived in \cite[Thm. 23.1.11]{Ismail} and it was shown that
\begin{equation} \label{eq:1.11}
     a_{\vec{n},j} = \frac{\displaystyle \int x^{n_j} P_{\vec{n}}(x)\, d\mu_j(x)}
                            {\displaystyle \int x^{n_j-1} P_{\vec{n}-\vec{e}_j}(x)\, d\mu_j(x)}, 
\end{equation}
and
\begin{equation} \label{eq:1.12}
    b_{\vec{n},j} = \int xP_{\vec{n}}(x) Q_{\vec{n}+\vec{e}_j}(x)\, d\mu(x), 
\end{equation}
where $Q_{\vec{n}} = \sum_{j=1}^r A_{\vec{n},j} w_j$.
There are similar recurrence relations for the type I multiple orthogonal polynomials:
\begin{eqnarray} \label{eq:1.10Q}  
   x Q_{\vec{n}}(x) &=& Q_{\vec{n}-\vec{e}_1}(x) + b_{\vec{n}-\vec{e}_1,1}Q_{\vec{n}}(x) + \sum_{j=1}^r a_{\vec{n},j} Q_{\vec{n}+\vec{e}_j}(x), 
   \nonumber \\
      \vdots & &  \\
   x Q_{\vec{n}}(x) & = & Q_{\vec{n}-\vec{e}_r}(x) + b_{\vec{n}-\vec{e}_r,r}Q_{\vec{n}}(x) + \sum_{j=1}^r a_{\vec{n},j} Q_{\vec{n}+\vec{e}_j}(x).
   \nonumber
\end{eqnarray}
Observe that the same recurrence coefficients $(a_{\vec{n},1},\ldots,a_{\vec{n},r})$ are used but that there is a shift in the
other recurrence coefficients $(b_{\vec{n}-\vec{e}_j,j})_{1 \leq j \leq r}$.

In Section 2 we will show how to find these recurrence relations from a Riemann-Hilbert problem
for multiple orthogonal polynomials, which was given by Van Assche, Geronimo and Kuijlaars \cite{WVAGerKui}.
This approach is not new but merely serves to show that these nearest neighbor recurrence relations enter 
naturally. In Section 3 it will be shown that not every choice of recurrence coefficients $(a_{\vec{n},1},\ldots,a_{\vec{n},r})$ 
and $(b_{\vec{n},1},\ldots,b_{\vec{n},r})$ is possible: our main result is that the recurrence coefficients
satisfy a system of partial difference equations. The boundary conditions are given by the
recurrence coefficients of the orthogonal polynomials with measure $\mu_j$ whenever $\vec{n} = n \vec{e}_j$, i.e.,
when all the components in the multi-index $\vec{n}$ are zero except for the $j$-th component. As such the recurrence
coefficients of multiple orthogonal polynomials are the solution of a discrete integrable system on the
lattice $\mathbb{N}^r$. We will first work with the case $r=2$ to simplify the notation and then give the general result for $r\geq 2$. 
In Section 4 we will show how the nearest neighbor recurrence relations and the partial difference equations for the recurrence coefficients
can be used to give a simple proof of the Christoffel-Darboux formula for multiple orthogonal polynomials, which was first given by
Daems and Kuijlaars \cite{DaeKuij}; see \cite{AFM} for a Christoffel-Darboux formula for multiple orthogonal polynomials of mixed type.
 This Christoffel-Darboux formula plays an important role in the analysis of certain random matrices
\cite{BleKuij2} \cite{Kuijl} and non-intersecting Brownian motions \cite{DaeKuij2}. In Section 5 we will give several examples of multiple orthogonal polynomials and give their recurrence coefficients explicitly. For some of these examples the recurrence coefficients are given here for the first time.
 
\section{The Riemann-Hilbert problem}
In this section we will take $r=2$ to keep the notation simple. The extension to general $r$ is straightforward but
requires bigger matrices.
Suppose that the measures $\mu_1$ and $\mu_2$ are absolutely continuous with weights $w_1$ and $w_2$. 
If $w_1$ and $w_2$ are H\"older continuous, then we can formulate the following Riemann-Hilbert problem:
find a $3\times 3$ matrix function $Y$ such that
\begin{enumerate}
\item $Y$ is analytic on $\mathbb{C} \setminus \mathbb{R}$,
\item the limits $Y_{\pm}(x) = \lim_{\epsilon \to 0+} Y(x\pm i\epsilon)$ exist and
\begin{equation} \label{eq:2.1}
      Y_+(x) = Y_-(x) \begin{pmatrix}
                        1 & w_1(w) & w_2(x) \\
                        0 &  1  & 0  \\
                        0 & 0 & 1
                        \end{pmatrix}, \qquad x \in \mathbb{R}, 
\end{equation}  
\item For $z \to \infty$ one has
\begin{equation} \label{eq:2.2}
     Y(z) = \big( I + \O(1/z) \big) \begin{pmatrix} z^{n+m} & 0 & 0 \\ 0 & z^{-n} & 0 \\ 0 & 0 & z^{-m} \end{pmatrix}.
\end{equation}
\end{enumerate}
In case $w_1$ or $w_2$ are defined on bounded intervals or semi-infinite intervals one additionally needs to specify the behavior
of $Y$ near the endpoints of the intervals. In \cite{WVAGerKui} it was shown that this Riemann-Hilbert problem
has a unique solution in terms of type II multiple orthogonal polynomials when $(n,m)$, $(n-1,m)$ and $(n,m-1)$
are normal indices, i.e.,
\begin{equation}  \label{eq:2.3}
  Y = \begin{pmatrix}
      P_{n,m} & C(P_{n,m}w_1) & C(P_{n,m}w_2) \\
      c_1(n,m) P_{n-1,m} & c_1 C(P_{n-1,m}w_1) & c_1 C(P_{n-1,m}w_2) \\
      c_2(n,m) P_{n,m-1} & c_2 C(P_{n,m-1}w_1) & c_2 C(P_{n,m-1}w_2)
      \end{pmatrix}
\end{equation}
where the Cauchy transform is used
\[   C(Pw) = \frac{1}{2\pi i} \int \frac{P(x)w(x)}{x-z}\, dx  \]
and the constants $c_1$ and $c_2$ are given by
\begin{eqnarray*}
    \frac{1}{c_1(n,m)} = \frac{-1}{2\pi i} \int P_{n-1,m}(x) x^{n-1}w_1(x)\, dx, \label{eq:2.4} \\
    \frac{1}{c_2(n,m)} = \frac{-1}{2\pi i} \int P_{n,m-1}(x) x^{m-1}w_2(x)\, dx. \label{eq:2.5} 
\end{eqnarray*}
In case $\mu_1$ and $\mu_2$ are discrete measures on the set $A=\{x_1,x_2,x_3,\ldots\}$
one can formulate a Riemann-Hilbert problem for a meromorphic matrix function: find a $3 \times 3$ matrix function $Y$ such that:
\begin{enumerate}
\item $Y$ is meromorphic on $\mathbb{C}$ with poles of order one at the points in $A$,
\item the residue of $Y$ at $x_k$ is given by
\[    \lim_{z \to x_k}   Y(z) \begin{pmatrix} 0 & \mu_1(\{x_k\}) & \mu_2(\{x_k\}) \\ 
                                              0 & 0 & 0 \\ 0 & 0 & 0 \end{pmatrix}  \]
\item $Y$ has the asymptotic condition \eqref{eq:2.2}.
\end{enumerate}
The solution of this problem is 
\[    Y(z) = \begin{pmatrix}
      P_{n,m}(z) & \sum_{k=1}^\infty \frac{P_{n,m}(x_k)}{z-x_k} \mu_1(\{x_k\}) 
              & \sum_{k=1}^\infty \frac{P_{n,m}(x_k)}{z-x_k} \mu_2(\{x_k\}) \\
      \hat{c}_1(n,m) P_{n-1,m}(z) & \hat{c}_1 \sum_{k=1}^\infty \frac{P_{n-1,m}(x_k)}{z-x_k} \mu_1(\{x_k\}) 
                            & \hat{c}_1 \sum_{k=1}^\infty \frac{P_{n-1,m}(x_k)}{z-x_k} \mu_2(\{x_k\}) \\
      \hat{c}_2(n,m) P_{n,m-1}(z) & \hat{c}_2 \sum_{k=1}^\infty \frac{P_{n,m-1}(x_k)}{z-x_k} \mu_1(\{x_k\}) 
                            & \hat{c}_2 \sum_{k=1}^\infty \frac{P_{n,m-1}(x_k)}{z-x_k} \mu_2(\{x_k\}) 
      \end{pmatrix}
\]
where
\begin{eqnarray*}
    \frac{1}{\hat{c}_1(n,m)} =  \sum_{k=1}^\infty P_{n-1,m}(x_k) x_k^{n-1} \mu_1(\{x_k\}) \\
    \frac{1}{\hat{c}_2(n,m)} =  \sum_{k=1}^\infty P_{n,m-1}(x_k) x_k^{m-1} \mu_2(\{x_k\}).  
\end{eqnarray*}
We will derive the recurrence relations for type II multiple orthogonal polynomials
from the Riemann-Hilbert problem.  The result is not limited to H\"older continuous weights $w_1$ and
$w_2$ but holds in general and can be derived from the biorthogonality \cite[Thm.~23.1.6]{Ismail}
\[   \int P_{\vec{n}}(x) Q_{\vec{m}}(w)\, d\mu(x) 
  = \begin{cases} 0, & \textrm{if $\vec{m} \leq \vec{n}$,} \\
                  0, & \textrm{if $|\vec{n}| \leq |\vec{m}|-2$,} \\
                  1, & \textrm{if $|\vec{m}|=|\vec{n}|+1$.}
    \end{cases}  \]
We merely use the Riemann-Hilbert problem to indicate that the nearest-neighbor recurrence relations
come out in a natural way. 

\begin{theorem} \label{thm:1}
Suppose all multi-indices $(n,m) \in \mathbb{N}^2$ are normal.
Then the type II multiple orthogonal polynomials satisfy the system of recurrence
relations
\begin{eqnarray}
    P_{n+1,m}(x) = (x-c_{n,m})P_{n,m}(x) - a_{n,m} P_{n-1,m}(x) - b_{n,m} P_{n,m-1}(x), \label{eq:2.6} \\
    P_{n,m+1}(x) = (x-d_{n,m})P_{n,m}(x) - a_{n,m} P_{n-1,m}(x) - b_{n,m} P_{n,m-1}(x), \label{eq:2.7}
\end{eqnarray}
with $a_{0,m}=0$ and $b_{n,0}=0$ for all $n,m \geq 0$.
\end{theorem}

\begin{proof}
First of all we observe that $\det Y$ is an analytic function in $\mathbb{C} \setminus \mathbb{R}$
which has no jump on the real axis, hence $\det Y$ is an entire function. Its behavior near infinity
is $\det Y(z) = 1 + \O(1/z)$ hence by Liouville's theorem we find that $\det Y = 1$. We can therefore consider
the matrix
\[     R_1(n,m) = Y_{n+1,m} Y_{n,m}^{-1}, \]
where the subscript $(n,m)$ is used for the Riemann-Hilbert problem with the type II multiple orthogonal
polynomial $P_{n,m}$ in the entry of the first row and first column of $Y_{n,m}$.
Clearly $R_1$ is an analytic function on $\mathbb{C} \setminus \mathbb{R}$, and since $Y_{n,m}$
and $Y_{n+1,m}$ have the same jump matrix on $\mathbb{R}$ we see that $R_1$ has no jump on $\mathbb{R}$.
Hence $R_1$ is an entire matrix function. If we write the asymptotic condition \eqref{eq:2.2} as
\[    Y_{n,m}(z) = \left( I + \frac{A(n,m)}{z} + \O(1/z^2) \right) 
		\begin{pmatrix} z^{n+m} & 0 & 0 \\ 0 & z^{-n} & 0 \\ 0 & 0 & z^{-m} \end{pmatrix}, \]
where $A(n,m)$ is the $3\times 3$ matrix coefficient of $1/z$ in the $\O(1/z)$ term of \eqref{eq:2.2}, then after some calculus we find
\[    R_1(n,m) = \begin{pmatrix}
                  z+A_{1,1}(n+1,m)-A_{1,1}(n,m)  & -A_{1,2}(n,m) & -A_{1,3}(n,m)  \\
                  A_{2,1}(n+1,m)  & 0 & 0 \\
                  A_{3,1}(n+1,m)  & 0 & 1  
                 \end{pmatrix} + \O(1/z), \]
where $A_{i,j}(n,m)$ is the entry on row $i$ and column $j$ of $A(n,m)$.
Liouville's theorem then implies that $R_1$ is the matrix polynomial
\begin{equation} \label{eq:2.8}
    R_1(n,m) = \begin{pmatrix}
                  z+A_{1,1}(n+1,m)-A_{1,1}(n,m) & -A_{1,2}(n,m) & -A_{1,3}(n,m) \\
                  A_{2,1}(n+1,m) & 0 &  0 \\
                  A_{3,1}(n+1,m) & 0 &  1  
                 \end{pmatrix},
\end{equation}
and we can therefore write
\begin{equation} \label{eq:2.9}
      Y_{n+1,m} = R_1(n,m) Y_{n,m}.
\end{equation}
In a similar way we also have
\begin{equation} \label{eq:2.10}
      Y_{n,m+1} = R_2(n,m) Y_{n,m}, 
\end{equation}
with
\begin{equation} \label{eq:2.11}
    R_2(n,m) = \begin{pmatrix}
                  z+A_{1,1}(n,m+1)-A_{1,1}(n,m) & -A_{1,2}(n,m) & -A_{1,3}(n,m) \\
                  A_{2,1}(n,m+1) & 1 &  0 \\
                  A_{3,1}(n,m+1) & 0 &  0  
                 \end{pmatrix}.
\end{equation}
 If we now work out the $(1,1)$-entry of \eqref{eq:2.9}, then we find \eqref{eq:2.6} with
\[    c_{n,m} = A_{1,1}(n,m)-A_{1,1}(n+1,m)  \]
and
\[    a_{n,m} = c_1(n,m)A_{1,2}(n,m), \quad b_{n,m} = c_2(n,m)A_{1,3}(n,m). \]
Similarly, if we work out the $(1,1)$-entry of \eqref{eq:2.10}, then we find \eqref{eq:2.7} with
\[    d_{n,m} = A_{1,1}(n,m)-A_{1,1}(n,m+1).  \]
\end{proof}

Observe that $\det R_1(n,m) = 1$ which implies that $A_{1,2}(n,m)A_{2,1}(n+1,m)=1$.
If we work out the $(2,1)$-entry of \eqref{eq:2.9} then we find
\[     c_1(n+1,m) P_{n,m} = A_{2,1}(n+1,m) P_{n,m}, \]
so that $c_1(n+1,m)=A_{2,1}(n+1,m)$.
Combined with the previous result, this gives
\[    a_{n,m} = \frac{c_1(n,m)}{c_1(n+1,m)}, \]
which corresponds to \eqref{eq:1.11} when $r=2$.
Similarly the identity $\det R_2(n,m)=1$ implies that $A_{1,3}(n,m)A_{3,1}(n,m+1)=1$ and the $(3,1)$-entry
of \eqref{eq:2.10} gives $c_2(n,m+1)=A_{3,1}(n,m+1)$, so that
\[    b_{n,m} = \frac{c_2(n,m)}{c_2(n,m+1)}.  \]

\begin{theorem} \label{thm:2.2}
Denote the moments by 
\[   \nu_k^{(j)} = \int x^k w_j(x)\, dx \]
and the moment matrix by
\begin{equation} \label{eq:2.12}
   M_{n,m} = \left( \begin{matrix}
                   \nu_0^{(1)} & \nu_1^{(1)} & \cdots & \nu_{n-1}^{(1)} \\
                   \nu_1^{(1)} & \nu_2^{(1)} & \cdots & \nu_{n}^{(1)} \\
	            \vdots & \vdots & \cdots & \vdots \\
                   \nu_{n+m-1}^{(1)} & \nu_{n+m}^{(1)} & \cdots & \nu_{2n+m-2}^{(1)} \end{matrix}		
           \begin{matrix}
                   \nu_0^{(2)} & \nu_1^{(2)} & \cdots & \nu_{m-1}^{(2)} \\
                   \nu_1^{(2)} & \nu_2^{(2)} & \cdots & \nu_{m}^{(2)} \\
	            \vdots & \vdots & \cdots & \vdots \\
                   \nu_{n+m-1}^{(2)} & \nu_{n+m}^{(2)} & \cdots & \nu_{n+2m-2}^{(2)} 		
                  \end{matrix} \right). 
\end{equation}
Then the recurrence coefficients can be written as 
\[    a_{n,m} = \frac{\det M_{n+1,m} \det M_{n-1,m}}{\big(\det M_{n,m}\big)^2}. \]
and
\[     b_{n,m} = \frac{\det M_{n,m+1} \det M_{n,m-1}}{\big(\det M_{n,m}\big)^2}. \]
Furthermore we have
\[   d_{n,m}-c_{n,m} = \frac{\det M_{n,m} \det M_{n+1,m+1}}{\det M_{n+1,m} \det M_{n,m+1}}.  \]
\end{theorem}

\begin{proof}
The type II multiple orthogonal polynomial can be written as
\[    P_{n,m}(x) = \frac{1}{\det M_{n,m}}
          \left|\begin{matrix}
                   \nu_0^{(1)} & \nu_1^{(1)} & \cdots & \nu_{n-1}^{(1)} \\
                   \nu_1^{(1)} & \nu_2^{(1)} & \cdots & \nu_{n}^{(1)} \\
	            \vdots & \vdots & \cdots & \vdots \\
                   \nu_{n+m}^{(1)} & \nu_{n+m+1}^{(1)} & \cdots & \nu_{2n+m-1}^{(1)} \end{matrix}		
           \begin{matrix}
                   \nu_0^{(2)} & \nu_1^{(2)} & \cdots & \nu_{m-1}^{(2)} \\
                   \nu_1^{(2)} & \nu_2^{(2)} & \cdots & \nu_{m}^{(2)} \\
	            \vdots & \vdots & \cdots & \vdots \\
                   \nu_{n+m}^{(2)} & \nu_{n+m+1}^{(2)} & \cdots & \nu_{n+2m-1}^{(2)} 		
                  \end{matrix} 
          \begin{matrix} 1 \\ x \\ \vdots \\ x^{n+m} \end{matrix} \right|     ,
\]
where $M_{n,m}$ is the moment matrix given in \eqref{eq:2.12}. This formula indeed has all the
orthogonality properties \eqref{eq:1.9}, and from it one easily finds
\[   \int x^n P_{n,m}(x)w_1(x)\, dx = \frac{(-1)^m \det M_{n+1,m}}{\det M_{n,m}}, \]
so that
\[    a_{n,m} = \frac{\det M_{n+1,m} \det M_{n-1,m}}{\big(\det M_{n,m}\big)^2}. \]
In a similar way
\[    \int x^m P_{n,m}(x)w_2(x)\, dx = \frac{\det M_{n,m+1}}{\det M_{n,m}} \]
and
\[    b_{n,m} = \frac{\det M_{n,m+1} \det M_{n,m-1}}{\big(\det M_{n,m}\big)^2}. \]
Take the $(3,1)$-entry of \eqref{eq:2.9} to find
\[    c_2(n+1,m) P_{n+1,m-1} = A_{3,1}(n+1,m) P_{n,m} + c_2(n,m) P_{n,m-1}. \]
On the other hand, if we subtract \eqref{eq:2.6} from \eqref{eq:2.7} then we have
\[     P_{n,m+1}-P_{n+1,m} = (c_{n,m}-d_{n,m}) P_{n,m}. \]
If we increase $m$ by one in the penultimate formula, then comparison with the previous formula gives
\begin{equation}  \label{d-c2}
      d_{n,m}-c_{n,m} = \frac{c_2(n,m+1)}{c_2(n+1,m+1)}
                        = \frac{\det M_{n,m} \det M_{n+1,m+1}}{\det M_{n+1,m} \det M_{n,m+1}}.  
\end{equation}
Similarly, if we take the $(2,1)$-entry of \eqref{eq:2.10}, then
\[     c_1(n,m+1)P_{n-1,m+1} = A_{2,1}(n,m+1)P_{n,m} + c_1(n,m) P_{n-1,m}, \]
and if we increase $n$ by one and compare with the previous formulas, then
\begin{equation}  \label{d-c1}
    d_{n,m}-c_{n,m} = - \frac{c_1(n+1,m)}{c_1(n+1,m+1)} 
		       =  \frac{\det M_{n,m} \det M_{n+1,m+1}}{\det M_{n+1,m} \det M_{n,m+1}}. 
\end{equation}
\end{proof}

Observe that the arrays $c$ and $d$ can be obtained if one knows the array $A_{1,1}$ and the latter contains
the second coefficients of the type II multiple orthogonal polynomials
\[   P_{n,m}(x) = x^{n+m} + A_{1,1}(n,m) x^{n+m-1} + \cdots. \]
Each $A_{1,1}(n,m)$ can also be expressed as the ratio of two moment matrices, 
\[     A_{1,1}(n,m) = - \frac{\det \widehat{M}_{n,m}}{\det M_{n,m}} \]
and the moment matrix in the numerator is
\[  \widehat{M}_{n,m} \left( \begin{matrix}
                   \nu_0^{(1)} & \nu_1^{(1)} & \cdots & \nu_{n-1}^{(1)} \\
                   \nu_1^{(1)} & \nu_2^{(1)} & \cdots & \nu_{n}^{(1)} \\
	            \vdots & \vdots & \cdots & \vdots \\
                   \nu_{n+m-2}^{(1)} & \nu_{n+m-1}^{(1)} & \cdots & \nu_{2n+m-3}^{(1)} \\ 
                    \nu_{n+m}^{(1)} & \nu_{n+m+1}^{(1)} & \cdots & \nu_{2n+m-1}^{(1)}                    
                        \end{matrix}		
           \begin{matrix}
                   \nu_0^{(2)} & \nu_1^{(2)} & \cdots & \nu_{m-1}^{(2)} \\
                   \nu_1^{(2)} & \nu_2^{(2)} & \cdots & \nu_{m}^{(2)} \\
	            \vdots & \vdots & \cdots & \vdots \\
                   \nu_{n+m-2}^{(2)} & \nu_{n+m-1}^{(2)} & \cdots & \nu_{n+2m-3}^{(2)} \\
  		   \nu_{n+m}^{(2)} & \nu_{n+m+1}^{(2)} & \cdots & \nu_{n+2m-1}^{(2)} \\
                  \end{matrix} \right), \]
i.e., it is the moment matrix $M_{n,m}$
from \eqref{eq:2.12} but with the last row replaced by   
\[    \begin{matrix} \nu_{n+m}^{(1)} & \nu_{n+m+1}^{(1)} & \cdots & \nu_{2n+m-1}^{(1)} &
       \nu_{n+m}^{(2)} & \nu_{n+m+1}^{(2)} & \cdots & \nu_{n+2m-1}^{(2)} \end{matrix} \] 

\section{Partial difference equations}
We will first consider the case $r=2$ to simplify the notation. The case for general $r$ is handled in the second part of the section. 

The recurrence coefficients $a_{n,m},b_{n,m},c_{n,m},d_{n,m}$ are not arbitrary sequences. They satisfy
certain partial difference equations, which are given in the following theorem.

\begin{theorem} \label{thm1}
Suppose that all the indices $(n,m) \in \mathbb{N}^2$ are normal.
The recurrence coefficients in the recurrence relations \eqref{eq:2.6}--\eqref{eq:2.7} for
type II multiple orthogonal polynomials satisfy the following equations
\begin{eqnarray}
    d_{n+1,m}-d_{n,m} &=& c_{n,m+1} - c_{n,m},  \label{eq:3.1}  \\
    b_{n+1,m}-b_{n,m+1} + a_{n+1,m}-a_{n,m+1} &=& \det \begin{pmatrix} d_{n+1,m} & d_{n,m} \\ 
	                                                             c_{n,m+1} & c_{n,m} \end{pmatrix},
  \qquad\qquad \label{eq:3.2} \\
    \frac{a_{n,m+1}}{a_{n,m}} &=& \frac{c_{n,m}-d_{n,m}}{c_{n-1,m}-d_{n-1,m}},  \label{eq:3.3} \\
    \frac{b_{n+1,m}}{b_{n,m}} &=& \frac{c_{n,m}-d_{n,m}}{c_{n,m-1}-d_{n,m-1}}.  \label{eq:3.4}
\end{eqnarray}
\end{theorem}

\begin{proof}
We obtain these equations by observing that $Y_{n+1,m+1}$ can be obtained in two different ways from the information
in $Y_{n,m}$. First we can compute $Y_{n+1,m}$ by using \eqref{eq:2.6} and then we can use \eqref{eq:2.7} (with $n$
replaced by $n+1$) to find $Y_{n+1,m+1}$, i.e.
\[     Y_{n+1,m+1} = R_2(n+1,m)R_1(n,m) Y_{n,m}.  \]
On the other hand we can also compute $Y_{n,m+1}$ by using \eqref{eq:2.7} and then use \eqref{eq:2.6} (with $m$ replaced by $m+1$) to find $Y_{n+1,m+1}$,
\[     Y_{n+1,m+1} = R_1(n,m+1)R_2(n,m) Y_{n,m}.  \]
Since $\det Y=1$ this implies that
\begin{equation}  \label{eq:3.5}
       R_2(n+1,m)R_1(n,m) = R_1(n,m+1)R_2(n,m).
\end{equation}
Observe that
\[   R_1(n,m) = \begin{pmatrix}
                 z-c_{n,m} & -a_{n,m}/c_1(n,m) & -b_{n,m}/c_2(n,m) \\
                 c_1(n+1,m) & 0 & 0 \\
                 c_2(n+1,m) & 0 & 1
                \end{pmatrix}, \]
\[     R_2(n,m) = \begin{pmatrix}
                 z-d_{n,m} & -a_{n,m}/c_1(n,m) & -b_{n,m}/c_2(n,m) \\
                 c_1(n,m+1) & 1 & 0 \\
                 c_2(n,m+1) & 0 & 0
                \end{pmatrix}, \]
hence \eqref{eq:3.5} is equivalent to
\[   \begin{pmatrix}
      S_{1,1}(z) & S_{1,2}(z) & S_{1,3}(z) \\
      S_{2,1}(z) & 0 & 0 \\
      S_{3,1}(z) & 0 & 0
     \end{pmatrix} =
     \begin{pmatrix} 0 & 0 & 0 \\ 0 & 0 & 0 \\ 0 & 0 & 0 \end{pmatrix}  \]
where
\begin{eqnarray*}
  S_{1,1}(z) &=& (z-c_{n,m})(z-d_{n+1,m})-a_{n+1,m}-b_{n+1,m}-(z-d_{n,m})(z-c_{n,m+1}) \\
             && +\ a_{n,m+1}+b_{n,m+1} , \\
  S_{1,2}(z) &=& \displaystyle -(z-d_{n+1,m})\frac{a_{n,m}}{c_1(n,m)}  
                +(z-c_{n,m+1}) \frac{a_{n,m}}{c_1(n,m)} + \frac{a_{n,m+1}}{c_1(n,m+1)} \\
  S_{1,3}(z) &=&  \displaystyle -(z-d_{n+1,m})\frac{b_{n,m}}{c_2(n,m)}-\frac{b_{n+1,m}}{c_2(n+1,m)}
                + (z-c_{n,m+1}) \frac{b_{n,m}}{c_2(n,m)} \\
  S_{2,1}(z) &=& (z-c_{n,m})c_1(n+1,m+1) + c_1(n+1,m)-(z-d_{n,m})c_1(n+1,m+1) \\
  S_{3,1}(z) &=& (z-c_{n,m})c_2(n+1,m+1) - (z-d_{n,m})c_2(n+1,m+1)-c_2(n,m+1).
\end{eqnarray*}
Observe that $S_{1,1}$ is a polynomial of degree one and the remaining quantities $S_{1,2}, S_{1,3}, S_{2,1}$ and
$S_{3,1}$ are polynomials of degree zero. If we set the coefficient of $z$ in $S_{1,1}$ equal to zero, then
we find \eqref{eq:3.1}. If we put the constant term in $S_{1,1}$ equal to zero, then we find \eqref{eq:3.2}.
If we put $S_{1,2}$ equal to zero and use \eqref{d-c1}, then we find \eqref{eq:3.3}, and if we put $S_{1,3}$
equal to zero and use \eqref{d-c2}, then we find \eqref{eq:3.4}. If we put $S_{2,1}$ and $S_{3,1}$ equal to zero,
then we find \eqref{d-c1} and \eqref{d-c2}.
\end{proof}

If we introduce the difference operators $\Delta_1, \Delta_2, \nabla_1$ and $\nabla_2$ for which
\[   \Delta_1 f(n,m) = f(n+1,m)-f(n,m), \quad \Delta_2 f(n,m)= f(n,m+1)-f(n,m)  \]
\[   \nabla_1 f(n,m) = f(n,m)-f(n-1,m), \quad \nabla_2 f(n,m) = f(n,m)-f(n,m-1), \]
then we find the partial difference equations
\begin{eqnarray*}   
      \Delta_1 d &=& \Delta_2 c, \\
      \Delta_1 (a+b) - \Delta_2 (a+b) &=& \det \begin{pmatrix} 
					\Delta_1 d & d \\ \Delta_2 c & c 	
					\end{pmatrix} \\
      \Delta_2 \log a = \nabla_1 \log (c-d) , &\quad& \Delta_1 \log b = \nabla_2 \log (c-d).  
\end{eqnarray*}

The result for general $r$ has more partial difference relations. In fact for each pair $(i,j)$ with $i \neq j$ we have
partial difference relations as in Theorem \ref{thm1} for the partial difference operators $\Delta_i,\Delta_j$ and $\nabla_i,\nabla_j$
with
\[  \Delta_i f(\vec{n}) = f(\vec{n}+\vec{e}_i) - f(\vec{n}), \qquad \nabla_i f(\vec{n}) = f(\vec{n})-f(\vec{n}-\vec{e}_i). \]
We will not use the Riemann-Hilbert problem explicitly but some of the elements of the proof use ideas from the Riemann-Hilbert
approach, in particular the transfer matrices matrices $R_k(\vec{n})$ are inspired by the matrices $R_1(n,m)$ and $R_2(n,m)$ which we used
for the case $r=2$.  

\begin{theorem}  \label{thm:diffr}
Suppose all multi-indices $\vec{n} \in \mathbb{N}^r$ are normal. Suppose $1 \leq i \neq j \leq r$, then the recurrence coefficients for
the nearest neighbor recurrence relations (\ref{eq:1.10}) satisfy
\begin{eqnarray}
      b_{\vec{n}+\vec{e}_i,j} - b_{\vec{n},j} &=& b_{\vec{n}+\vec{e}_j,i} - b_{\vec{n},i},   \label{eq:r1} \\
      \sum_{k=1}^r a_{\vec{n}+\vec{e}_j,k} - \sum_{k=1}^r a_{\vec{n}+\vec{e}_i,k}
       &=& \det \begin{pmatrix}  b_{\vec{n}+\vec{e}_j,i} & b_{\vec{n},i} \\ b_{\vec{n}+\vec{e}_i,j} & b_{\vec{n},j} \end{pmatrix}, \label{eq:r2} \\
  \frac{a_{\vec{n},i}}{a_{\vec{n}+\vec{e}_j,i}} &=& \frac{b_{\vec{n}-\vec{e}_i,j}-b_{\vec{n}-\vec{e}_i,i}}{b_{\vec{n},j}-b_{\vec{n},i}}, \label{eq:r3} 
\end{eqnarray}
\end{theorem}

\begin{proof}
The recurrence relations (\ref{eq:1.10}) give a remarkable relation between three neighboring polynomials:
\[  P_{\vec{n}+\vec{e}_k}(x) - P_{\vec{n}+\vec{e}_j}(x) = (b_{\vec{n},j}-b_{\vec{n},k}) P_{\vec{n}}(x). \]
This relation and the $k$-th relation in (\ref{eq:1.10}) can be written as
\[  \begin{pmatrix}
    P_{\vec{n}+\vec{e}_k}(x) \\
    P_{\vec{n}+\vec{e}_k-\vec{e}_1}(x) \\
    P_{\vec{n}+\vec{e}_k-\vec{e}_2}(x) \\
    \vdots \\
    P_{\vec{n}}(x) \\
    \vdots \\
    P_{\vec{n}+\vec{e}_k-\vec{e}_r}(x)
    \end{pmatrix}
  = \begin{pmatrix}
    x-b_{\vec{n},k} & -a_{\vec{n},1} & -a_{\vec{n},2} & \cdots & -a_{\vec{n},k} & \cdots & -a_{\vec{n},r} \\
      1 & B_{k,1}(\vec{n}) & 0 & 0 & 0 & 0 & 0 \\
      1 &  0 & B_{k,2}(\vec{n})  & 0 & 0 & 0 & 0 \\
      \vdots & \\
      1 & 0 & 0 & 0 & 0 & 0 & 0  \\
      \vdots & \\
      1 & 0 & 0 & 0 & 0 & 0 & B_{k,r}(\vec{n})
      \end{pmatrix}
      \begin{pmatrix}
    P_{\vec{n}}(x) \\
    P_{\vec{n}-\vec{e}_1}(x) \\
    P_{\vec{n}-\vec{e}_2}(x) \\
    \vdots \\
    P_{\vec{n}-\vec{e}_k}(x) \\
    \vdots \\
    P_{\vec{n}-\vec{e}_r}(x)
    \end{pmatrix}.   \]
where $B_{k,j}(\vec{n})=b_{\vec{n}-\vec{e}_j,j}-b_{\vec{n}-\vec{e}_j,k}$ (note that $B_{k,k}(\vec{n}) = 0$). We will write this matrix relation as
\[   Y_{\vec{n}+\vec{e}_k} = R_k(\vec{n}) Y_{\vec{n}}. \]
Note that $Y_{\vec{n}}$ is now a vector containing the multiple orthogonal polynomial $P_{\vec{n}}$ and its neighbors from below, and not
a matrix as in the Riemann-Hilbert problem.
Now we observe that there are two ways to obtain $Y_{\vec{n}+\vec{e}_i+\vec{e}_j}$ when $i \neq j$. The first way is 
to first compute $Y_{\vec{n}+\vec{e}_i}$ from $Y_{\vec{n}}$ and then to compute $Y_{\vec{n}+\vec{e}_i+\vec{e}_j}$ from
$Y_{\vec{n}+\vec{e}_i}$:
\[     Y_{\vec{n}+\vec{e}_i+\vec{e}_j} = R_j(\vec{n}+\vec{e}_i) R_i(\vec{n}) Y_{\vec{n}}.  \]
The second way is to first compute $Y_{\vec{n}+\vec{e}_j}$ from $Y_{\vec{n}}$ and then to compute $Y_{\vec{n}+\vec{e}_i+\vec{e}_j}$ from
$Y_{\vec{n}+\vec{e}_j}$:
\[   Y_{\vec{n}+\vec{e}_i+\vec{e}_j} = R_i(\vec{n}+\vec{e}_j) R_j(\vec{n}) Y_{\vec{n}}. \]
The compatibility between these two ways implies that
\[   R_j(\vec{n}+\vec{e}_i) R_i(\vec{n}) Y_{\vec{n}} = R_i(\vec{n}+\vec{e}_j) R_j(\vec{n}) Y_{\vec{n}}. \]
The polynomials in the vector $Y_{\vec{n}}$ are linearly independent whenever $\vec{n}$ is a normal index. Indeed, if $(c_0,c_1,\ldots,c_r)$
are such that
\[  c_0 P_{\vec{n}} + \sum_{j=1}^r c_j P_{\vec{n}-\vec{e}_j} = 0, \]
then surely $c_0=0$ follows by comparing the leading coefficients. If we multiply by $x^{n_k-1}$ and integrate with respect to $\mu_k$, then
\[      c_k \int x^{n_k-1} P_{\vec{n}-\vec{e}_k}(x)\, d\mu_k(x) = 0  \]
so that $c_k=0$ follows, since the integral does not vanish. Indeed, if the integral vanishes, then $P_{\vec{n}} - a P_{\vec{n}-\vec{e}_k}$
is (for every $a$) a monic polynomial of degree $|\vec{n}|$ satisfying the orthogonality relations (\ref{eq:1.9}), which contradicts
the normality of $\vec{n}$. The linear independence of the polynomials in $Y_{\vec{n}}$ implies that
\begin{equation}  \label{RR}
   R_j(\vec{n}+\vec{e}_i) R_i(\vec{n}) = R_i(\vec{n}+\vec{e}_j) R_j(\vec{n}). 
\end{equation}
The transfer matrix $R_k(\vec{n})$ can be written as
\[    R_k(\vec{n}) = \begin{pmatrix}
               x-b_{\vec{n},k} & -\vec{a}_{\vec{n}}^{\,t} \\
               \vec{1} & B_k(\vec{n})
                \end{pmatrix}  \]
where $\vec{a}_{\vec{n}}$ and $\vec{1}$ are (column) vectors of length $r$ containing the $a_{\vec{n},j}$ ($1 \leq j \leq r$) and $r$ ones respectively,
and $B_k(\vec{n})$ is a diagonal matrix of size $r$ containing the $B_{k,j}(\vec{n})$ ($1 \leq j \leq r$) on the diagonal.
We then have
\begin{multline*}    R_i(\vec{n}+\vec{e}_j) R_j(\vec{n}) \\
    = \begin{pmatrix}
              (x-b_{\vec{n}+\vec{e}_j,i})(x-b_{\vec{n},j}) - \sum_{k=1}^r a_{\vec{n}+\vec{e}_j,k} &
                -(x-b_{\vec{n}+\vec{e}_j,i})\vec{a}_{\vec{n}}^{\,t} - \vec{a}_{\vec{n}+\vec{e}_j}^{\,t} B_j(\vec{n}) \\
      (x-b_{\vec{n},j}) \vec{1} + B_i(\vec{n}+\vec{e}_j)\vec{1} & -\vec{1}\cdot \vec{a}_{\vec{n}}^{\,t} + B_i(\vec{n}+\vec{e}_j) B_j(\vec{n})
         \end{pmatrix}  
\end{multline*}
and
\begin{multline*}    R_j(\vec{n}+\vec{e}_i) R_i(\vec{n}) \\
    = \begin{pmatrix}
              (x-b_{\vec{n}+\vec{e}_i,j})(x-b_{\vec{n},i}) - \sum_{k=1}^r a_{\vec{n}+\vec{e}_i,k} &
                -(x-b_{\vec{n}+\vec{e}_i,j})\vec{a}_{\vec{n}}^{\,t} - \vec{a}_{\vec{n}+\vec{e}_i}^{\,t} B_i(\vec{n}) \\
      (x-b_{\vec{n},i}) \vec{1} + B_j(\vec{n}+\vec{e}_i)\vec{1} & -\vec{1}\cdot \vec{a}_{\vec{n}}^{\,t} + B_j(\vec{n}+\vec{e}_i) B_i(\vec{n})
         \end{pmatrix} . 
\end{multline*}
If we compare the $(1,1)$-entries in (\ref{RR}) then this gives a polynomial of degree 1 which is identically zero.
The coefficient of $x$ vanishes whenever (\ref{eq:r1}) holds and the constant term vanishes whenever (\ref{eq:r2}) holds.
If we check the remaining entries on the first row of (\ref{RR}), then we have for $1 \leq k \leq r$
\begin{equation}  \label{aB}
   b_{\vec{n}+\vec{e}_i,j} a_{\vec{n},k} - B_{i,k}(\vec{n}) a_{\vec{n}+\vec{e}_i,k} = 
    b_{\vec{n}+\vec{e}_j,j} a_{\vec{n},k} - B_{j,k}(\vec{n}) a_{\vec{n}+\vec{e}_j,k} . 
\end{equation}
If we take this for $k=i$ then we find
\[   \frac{a_{\vec{n},i}}{a_{\vec{n}+\vec{e}_j,i}} = 
   \frac{b_{\vec{n}-\vec{e}_i,j}-b_{\vec{n}-\vec{e}_i,i}}{b_{\vec{n}+\vec{e}_i,j}-b_{\vec{n}+\vec{e}_j,i}}, \]
which gives (\ref{eq:r3}) after using (\ref{eq:r1}) to simplify the denominator on the right hand side.
By taking $k=j$ in (\ref{aB}) we find (\ref{eq:r3}) again but with $i$ and $j$ interchanged.
\end{proof}

The partial difference relations can be written as
\begin{eqnarray*}  
    \Delta_i b_{\vec{n},j} &=& \Delta_j b_{\vec{n},i},  \\
    \Delta_j \sum_{k=1}^r a_{\vec{n},k} - \Delta_i \sum_{k=1}^r a_{\vec{n},k} &=& \det \begin{pmatrix}
                                           \Delta_j b_{\vec{n},i} & b_{\vec{n},i} \\
					   \Delta_i b_{\vec{n},j} & b_{\vec{n},j}  \end{pmatrix}, \\
    \Delta_j \log a_{\vec{n},i} &=& \nabla_i \log \left( b_{\vec{n},j} - b_{\vec{n},i} \right), 
\end{eqnarray*}

\section{The Christoffel-Darboux formula}
Daems and Kuijlaars \cite{DaeKuij} have given the following Christoffel-Darboux formula for multiple orthogonal polynomials

\begin{theorem}
Suppose $(\vec{n}_i)_{i=0,1,\ldots,|\vec{n}|}$ is a path in $\mathbb{N}^r$ such that $\vec{n}_0 = \vec{0}$, $\vec{n}_{|\vec{n}|} = \vec{n}$
and for every $i \in \{0,1,\ldots,|\vec{n}|-1\}$ one has $\vec{n}_{i+1} - \vec{n}_i = \vec{e}_k$ for some $k \in \{1,2,\ldots,r\}$. Then
\begin{equation}  \label{CD}
 (x-y) \sum_{i=0}^{|\vec{n}|-1} P_{\vec{n}_i}(x) Q_{\vec{n}_{i+1}}(y)
    = P_{\vec{n}}(x) Q_{\vec{n}}(y) - \sum_{j=1}^r a_{\vec{n},j} P_{\vec{n}-\vec{e}_j}(x)Q_{\vec{n}+\vec{e}_j}(y).  
\end{equation}
\end{theorem}

Observe that the right hand side is independent of the path $(\vec{n}_i)_{i=0,1,\ldots,|\vec{n}|}$ in $\mathbb{N}^r$. 
A similar formule for multiple orthogonal polynomials of mixed type is given in \cite{AFM}. 
We will give a simple proof based on the nearest neighbor recurrence relations (\ref{eq:1.11}) and the partial difference equations
in Theorem \ref{thm:diffr}.
\begin{proof}
Use the recurrence relation (\ref{eq:1.10}) for $\vec{n}_i$ and $\vec{n}_i + \vec{e}_k = \vec{n}_{i+1}$ to find
\begin{equation}  \label{Pirec}
   xP_{\vec{n}_i}(x) = P_{\vec{n}_{i+1}}(x) + b_{\vec{n}_i,k} P_{\vec{n}_i}(x) +
    \sum_{j=1}^r a_{\vec{n}_i,j} P_{\vec{n}_i-\vec{e}_j}(x) . 
\end{equation}
In a similar way we can use the recurrence relation (\ref{eq:1.10Q}) for $\vec{n}_{i+1}$ and $\vec{n}_{i+1}-\vec{e}_k = \vec{n}_i$, but with $x$ replaced by $y$, to find
\begin{equation}  \label{Qi1rec}
   yQ_{\vec{n}_{i+1}}(y) = Q_{\vec{n}_i(}y) + b_{\vec{n}_i,k} Q_{\vec{n}_{i+1}}(y) +
    \sum_{j=1}^r a_{\vec{n}_{i+1},j} Q_{\vec{n}_{i+1}+\vec{e}_j}(y) . 
\end{equation}
Multiply (\ref{Pirec}) by $Q_{\vec{n}_{i+1}}(y)$ and (\ref{Qi1rec}) by $P_{\vec{n}_i}(x)$ and subtract both equations, to find
\begin{multline}  \label{PQ}
  (x-y) P_{\vec{n}_i}(x) Q_{\vec{n}_{i+1}}(y) = P_{\vec{n}_{i+1}}(x)Q_{\vec{n}_{i+1}}(y) - P_{\vec{n}_i}(x)Q_{\vec{n}_i}(y) \\
    + \sum_{j=1}^r a_{\vec{n}_i,j} P_{\vec{n}_i-\vec{e}_j}(x)Q_{\vec{n}_{i+1}}(y) 
    - \sum_{j=1}^r a_{\vec{n}_{i+1},j} P_{\vec{n}_i}(x)Q_{\vec{n}_{i+1}+\vec{e}_j}(y).
\end{multline} 
If we sum this for $i=0$ to $|\vec{n}|-1$, then the left hand side corresponds to the left hand side in (\ref{CD}), and the first terms on the right
hand side give
\[   \sum_{i=0}^{|\vec{n}|-1} \bigl( P_{\vec{n}_{i+1}}(x)Q_{\vec{n}_{i+1}}(y) - P_{\vec{n}_i}(x)Q_{\vec{n}_i}(y) \bigr) =
 P_{\vec{n}}(x)Q_{\vec{n}}(y) \]
by using the telescoping property and $Q_{\vec{0}}=0$. The two sums on the right hand side of (\ref{PQ}) need to be rewritten a bit.
Observe that the recurrence relation (\ref{eq:1.10}) implies that
\[    P_{\vec{n}+\vec{e}_k}(x) - P_{\vec{n}+\vec{e}_j}(x) = (b_{\vec{n},j}-b_{\vec{n},k}) P_{\vec{n}}(x).  \]
If we use this for $\vec{n} = \vec{n}_i-\vec{e}_j$, then this gives
\begin{equation}   \label{P1}
   P_{\vec{n}_i}(x) = P_{\vec{n}_{i+1}-\vec{e}_j}(x) + (b_{\vec{n}_i-\vec{e}_j,k}-b_{\vec{n}_i-\vec{e}_j,j}) P_{\vec{n}_i-\vec{e}_j}(x).
\end{equation}
In a similar way, the recurrence relation (\ref{eq:1.10Q}) implies
\[    Q_{\vec{n}-\vec{e}_k}(y) -Q_{\vec{n}-\vec{e}_j}(y) = (b_{\vec{n}-\vec{e}_j,j}-b_{\vec{n}-\vec{e}_k,k}) Q_{\vec{n}}(y). \]
If we use this for $\vec{n} = \vec{n}_{i+1}+\vec{e}_j$, then we find
\begin{equation} \label{Q1}
    Q_{\vec{n}_{i+1}}(y) = Q_{\vec{n}_i+\vec{e}_j}(y) + (b_{\vec{n}_i+\vec{e}_j,k} - b_{\vec{n}_{i+1},j}) Q_{\vec{n}_{i+1}+\vec{e}_j}(y).
\end{equation}
Replace $P_{\vec{n}_i}(x)$ in the second sum of (\ref{PQ}) by (\ref{P1}) and replace $Q_{\vec{n}_{i+1}}(y)$ in the first sum of
(\ref{PQ}) by (\ref{Q1}), then we find that
\begin{eqnarray*}
   \lefteqn{\sum_{j=1}^r a_{\vec{n}_i,j} P_{\vec{n}_i-\vec{e}_j}(x)Q_{\vec{n}_{i+1}}(y) 
    - \sum_{j=1}^r a_{\vec{n}_{i+1},j} P_{\vec{n}_i}(x)Q_{\vec{n}_{i+1}+\vec{e}_j}(y)} & & \\
   &=& \sum_{j=1}^r a_{\vec{n}_i,j} P_{\vec{n}_i-\vec{e}_j}(x)Q_{\vec{n}_i+\vec{e}_j}(y) 
    - \sum_{j=1}^r a_{\vec{n}_{i+1},j} P_{\vec{n}_{i+1}-\vec{e}_j}(x)Q_{\vec{n}_{i+1}+\vec{e}_j}(y) \\
   & &  + \sum_{j=1}^r P_{\vec{n}_i-\vec{e}_j}(x) Q_{\vec{n}_{i+1}+\vec{e}_j}(y) \bigl[
       a_{\vec{n}_i,j} (b_{\vec{n}_i+\vec{e}_j,k} - b_{\vec{n}_{i+1},j}) - a_{\vec{n}_{i+1},j}
      (b_{\vec{n}_i-\vec{e}_j,k}-b_{\vec{n}_i-\vec{e}_j,j}) \bigr].
\end{eqnarray*}
The last sum on the right hand side vanishes because of the partial difference relations in Theorem \ref{thm:diffr}, in particular
(\ref{eq:r3}) and (\ref{eq:r1}) but with $i=k$. The other two sums
are now of the form $S_i-S_{i+1}$, so that we are summing a telescoping series, giving
\[  \sum_{i=0}^{|\vec{n}|-1} (S_i-S_{i+1}) = S_0 - S_{|\vec{n}|} = - \sum_{j=1}^r a_{\vec{n},j} P_{\vec{n}-\vec{e}_j}(x)Q_{\vec{n}+\vec{e}_j}(y), \]
which is the sum on the right hand side of (\ref{CD}).
\end{proof}

\section{Some examples}

In this section we give several examples of multiple orthogonal polynomials and
their nearest neighbor recurrence coefficients. For each example one can verify that the
partial difference equations in Theorem \ref{thm:diffr} indeed hold. Some of these examples
were given before, e.g., in \cite{AptBraWVA}, \cite{ArCoWVA}, \cite{BChDou}, \cite{BleKuij}, \cite{JCWVA}, \cite[Chap.\ 23]{Ismail}, \cite{WVAEC}, but the explicit expression for the recurrence coefficients appears here
for the first time for most of these families. 

\subsection{Multiple Hermite polynomials}
See \cite[\S 3.4]{WVAEC}, \cite[\S 23.5]{Ismail}, \cite{BleKuij}.
These are given by
\[  \int_{-\infty}^\infty x^k H_{\vec{n}}(x) e^{-x^2+c_j x}\, dx = 0, \qquad k=0,1,\ldots, n_j-1, \]
for $1 \leq j \leq r$, where $c_i \neq c_j$ whenever $i \neq j$. The recurrence relation is explicitly given as
\[    x H_{\vec{n}}(x) = H_{\vec{n}+\vec{e}_k}(x) + \frac{c_k}{2}H_{\vec{n}}(x) + \frac12 \sum_{j=1}^r n_j H_{\vec{n}-\vec{e}_j}(x), \]
for $1 \leq k \leq r$, so that
\[    b_{\vec{n},j} = c_j/2, \quad a_{\vec{n},j} = n_j/2, \qquad 1 \leq j \leq r. \]

\subsection{Multiple Charlier polynomials}
See \cite[\S 4.1]{ArCoWVA}, \cite[\S 23.6.1]{Ismail}.
The orthogonality relations are
\[  \sum_{k=0}^\infty C_{\vec{n}}(k) k^\ell \frac{a_j^k}{k!} = 0, \qquad \ell=0,1,\ldots, n_j-1, \]
for $1 \leq j \leq r$, where $a_i >0$ and $a_i \neq a_j$ whenever $i \neq j$. The recurrence relation is given by
\[   xC_{\vec{n}}(x) = C_{\vec{n}+\vec{e}_k}(x) + (a_k + |\vec{n}|) C_{\vec{n}}(x) + \sum_{j=1}^r n_j a_j C_{\vec{n}-\vec{e}_j}(x), \]
so that
\[    b_{\vec{n},j} = |\vec{n}| + a_j, \quad a_{\vec{n},j} = n_ja_j, \qquad 1 \leq j \leq r. \]

\subsection{Multiple Laguerre polynomials of the first kind}  \label{sec:ML1}
See \cite[\S 3.2]{WVAEC} and \cite[\S 23.4.1]{Ismail}, \cite{BleKuij}.
These are given by the orthogonality relations
\[  \int_0^\infty x^k L_{\vec{n}}(x) x^{\alpha_j} e^{-x}\, dx = 0 , \qquad k = 0, 1, \ldots, n_j-1, \]
for $1 \leq j \leq r$, where $\alpha_1, \ldots, \alpha_r > -1$ and $\alpha_i - \alpha_j \notin \mathbb{Z}$. They can be obtained using the Rodrigues formula
\begin{equation}  \label{MLagIRod}
 (-1)^{|\vec{n}|} e^{-x} L_{\vec{n}}(x) 
 = \prod_{j=1}^r \left( x^{-\alpha_j} \frac{d^{n_j}}{dx^{n_j}} x^{n_j+\alpha_j} \right) e^{-x}  
\end{equation}
where the product of the differential operators can be taken in any order. This Rodrigues formula is
useful for computing the recurrence coefficients. Indeed, we easily find
\[    \int_0^\infty x^{n_j} L_{\vec{n}}(x) x^{\alpha_j} e^{-x} =
    (-1)^{|\vec{n}|} \int_0^\infty x^{n_j+\alpha_j-\alpha_1} \frac{d^{n_1}}{dx^{n_1}} x^{n_1+\alpha_1}
      \prod_{i=2}^r  \left( x^{-\alpha_i} \frac{d^{n_i}}{dx^{n_i}} x^{n_i+\alpha_i} \right) e^{-x}\, dx \]
and integration by parts ($n_1$ times) gives
\begin{eqnarray*}
   &=&
    (-1)^{|\vec{n}|+n_1} \int_0^\infty \left( \frac{d^{n_1}}{dx^{n_1}} x^{n_j+\alpha_j-\alpha_1}\right)  
     x^{n_1+\alpha_1}  
     \prod_{i=2}^r  \left( x^{-\alpha_i} \frac{d^{n_i}}{dx^{n_i}} x^{n_i+\alpha_i} \right) e^{-x}\, dx \\
    &=& (-1)^{|\vec{n}|+n_1} \binom{n_j+\alpha_j-\alpha_1}{n_1} n_1! \int_0^\infty x^{n_j+\alpha_j}  
     \prod_{i=2}^r  \left( x^{-\alpha_i} \frac{d^{n_i}}{dx^{n_i}} x^{n_i+\alpha_i} \right) e^{-x}\, dx.
\end{eqnarray*}
Repeating this $r$ times gives
\begin{eqnarray*} 
 \int_0^\infty x^{n_j} L_{\vec{n}}(x) x^{\alpha_j} e^{-x} &=&
    \prod_{i=1}^r \binom{n_j+\alpha_j-\alpha_i}{n_i} n_i! \int_0^\infty x^{n_j+\alpha_j} e^{-x}\, dx \nonumber \\
    &=& \Gamma(n_j+\alpha_j+1) \prod_{i=1}^r \binom{n_j+\alpha_j-\alpha_i}{n_i} n_i!.  
\end{eqnarray*}
If we then use (\ref{eq:1.11}) then we find
\begin{equation}  \label{anjMLagI}
    a_{\vec{n},j} = n_j (n_j+\alpha_j) \prod_{i=1, i\neq j}^r \frac{n_j+\alpha_j-\alpha_i}{n_j-n_i+\alpha_j-\alpha_i},
   \qquad j=1,\ldots,r.
\end{equation}
The recurrence coefficients $b_{\vec{n},k}$ can be obtained by comparing the coefficients of $x^{|\vec{n}|}$ in
the recurrence relation (\ref{eq:1.10}). They are
\begin{equation}  \label{bnkMLagI}
    b_{\vec{n},k} = |\vec{n}| + n_k + \alpha_k + 1, \qquad k=1,\ldots,r.  
\end{equation}
These recurrence coefficients are here given for the first time.
Observe that for $r=1$ one retrieves the well known recurrence coefficients for the monic Laguerre polynomials
$a_n^2 = n(n+\alpha)$ and $b_n = 2n+\alpha+1$.

An interesting observation is that these recurrence coefficients of multiple Laguerre polynomials of the first kind 
satisfy
\begin{equation} \label{sumanjMLagI}
   \sum_{j=1}^r a_{\vec{n},j} = \sum_{j=1}^r n_j \alpha_j 
  + \frac12 \left( |\vec{n}|^2 + \sum_{j=1}^r n_j^2 \right) > 0. 
\end{equation}
Indeed, if we introduce the polynomials
\[   q_r(x) = \prod_{j=1}^r (x-\alpha_j), \qquad Q_{r,\vec{n}}(x) = \prod_{j=1}^r (x-n_j-\alpha_j), \]
then $a_{\vec{n},j} = (n_j+\alpha_j) q_r(n_j+\alpha_j)/Q_{r,\vec{n}}'(n_j+\alpha_j)$ and hence by the residue theorem
\[   \sum_{j=1}^r a_{\vec{n},j} = \frac{1}{2\pi i} \int_\Gamma \frac{zq_r(z)}{Q_{r,\vec{n}}(z)}\, dz \]
where $\Gamma$ is a closed contour around all the zeros $n_i+\alpha_i$ $(i=1,\ldots,r)$ of $Q_{r,\vec{n}}$. We can take
for instance the circle $C_R= \{R e^{i\theta}\ : \ \theta \in [0,2\pi]\}$ for $R$ large enough.
If we change the variable $z = 1/\xi$, then the integral becomes
\[  \frac{1}{2\pi i} \int_{C_r} \frac{q_r^*(\xi)}{\xi^3 Q_{r,\vec{n}}^*(\xi)}\, d\xi  \]
where $C_r$ is the circle with radius $r=1/R$ and the orientation is positive by changing the sign appropriately.
Here $q_r^*(z)=z^r q_r(1/z)$ and $Q_{r,\vec{n}}^*(z) = z^r Q_{r,\vec{n}}(1/z)$ are the polynomials taking in reversed order.
The residue theorem now tells us that the integral is the residue of $q_r^*(\xi)/\xi^3 Q_{r,\vec{n}}^*(\xi)$ at $\xi=0$, and this
can easily be computed by using
\[  \frac{1-\xi \alpha_j}{1-\xi (n_j+\alpha_j)} = 1 + \xi n_j +\xi^2 (n_j+\alpha_j) + \mathcal{O}(\xi^3) , \]
so that 
\[  \frac{q_r^*(\xi)}{Q_{r,\vec{n}}^*(\xi)} = 1 + \xi |\vec{n}| + 
  \xi^2 \left( \sum_{j=1}^r n_j (n_j+\alpha_j) + \sum_{i<j}^r n_in_j \right) + \mathcal{O}(\xi^3), \]
and the residue of $q_r^*(\xi)/\xi^3 Q_{r,\vec{n}}^*(\xi)$ thus corresponds to the coefficient of the quadratic term,
giving (\ref{sumanjMLagI}).

\subsection{Multiple Laguerre polynomials of the second kind}
See \cite[Remark 5 on p.~160]{NikSor}, \cite[\S 3.3]{WVAEC} and \cite[\S 23.4.2]{Ismail}, \cite{BleKuij}.
These are given by the orthogonality relations
\[  \int_0^\infty x^k L_{\vec{n}}(x) x^{\alpha} e^{-c_j x}\, dx = 0 , \qquad k = 0, 1, \ldots, n_j-1, \]
for $1 \leq j \leq r$, where $\alpha > -1$, $c_1, \ldots, c_r > 0$ and $c_i \neq c_j$ whenever $i \neq j$. 
They can be obtained using the Rodrigues formula
\begin{equation}  \label{MLagIIRod}
 (-1)^{|\vec{n}|} \left( \prod_{j=1}^r c_j^{n_j} \right) x^\alpha L_{\vec{n}}(x) 
 = \prod_{j=1}^r \left(  e^{c_jx} \frac{d^{n_j}}{dx^{n_j}} e^{-c_jx}  \right) x^{|\vec{n}|+\alpha}  
\end{equation}
where the differential operators in the product can be taken in any order.
A useful integral is
\[   \int_0^\infty e^{-\lambda x}  \prod_{j=1}^r \left(  e^{c_jx} \frac{d^{n_j}}{dx^{n_j}} e^{-c_jx}  \right) x^{|\vec{n}|+\alpha} \, dx
    = (-1)^{|\vec{n}|} \frac{\Gamma(|\vec{n}|+\alpha+1)}{\lambda^{|\vec{n}|+\alpha+1}} \prod_{j=1}^r (c_j-\lambda)^{n_j} \]
which can be evaluated by using integration by parts in a similar way as in the previous example.
Observe that the right hand side has a zero at $\lambda = c_j$ of multiplicity $n_j$. Using (\ref{MLagIIRod}) we thus have for $\lambda > 0$
\[  \int_0^\infty e^{-\lambda x} x^\alpha L_{\vec{n}}(x)\, dx  =  \frac{\Gamma(|\vec{n}|+\alpha+1)}{\lambda^{|\vec{n}|+\alpha+1}}
  \prod_{i=1}^r ( 1-\lambda/c_i)^{n_i} . \]
Clearly
\[   \left. \frac{d^k}{d\lambda^k} \int_0^\infty e^{-\lambda x} x^\alpha L_{\vec{n}}(x)\, dx\right|_{\lambda=c_j} 
= (-1)^k \int_0^\infty x^k e^{c_jx} x^\alpha L_{\vec{n}}(x)\, dx
   = 0, \qquad 0 \leq k < n_j, \]
which confirms the orthogonality relations, and for $k=n_j$
\[       \int_0^\infty x^{n_j} e^{c_jx} x^\alpha L_{\vec{n}}(x)\, dx = 
    \frac{\Gamma(|\vec{n}|+\alpha+1)}{c_j^{|\vec{n}|+n_j+\alpha+1}} n_j! \prod_{i=1,i\neq j}^r \left( 1 - \frac{c_j}{c_i} \right).  \]
If we then use (\ref{eq:1.11}) then we find
\begin{equation}  \label{anjMLagII}
  a_{\vec{n},j} = \frac{(|\vec{n}|+\alpha)n_j}{c_j^2}, \qquad 1 \leq j \leq r. 
\end{equation}
For the coefficients $b_{\vec{n},k}$ we compare the coefficients of $x^{|\vec{n}|}$ on both sides of the recurrence relation (\ref{eq:1.10}) and
use the explicit expression (23.4.5) in \cite{Ismail} to find
\begin{equation}  \label{bnjMLagII}
     b_{\vec{n},k} = \frac{|\vec{n}|+\alpha+1}{c_k} + \sum_{j=1}^r \frac{n_j}{c_j}.  
\end{equation}
Observe that for $r=1$ and $c_1=1$ we retrieve the recurrence coefficients of monic Laguerre polynomials
$b_n = 2n+\alpha+1$ and $a_n^2 = n(n+\alpha)$.

\subsection{Jacobi-Pi\~neiro polynomials}
See \cite[Remark 7 on p.~162]{NikSor}, \cite[\S 2.1]{WVAEC} and \cite[\S 23.3.2]{Ismail}.
These are multiple orthogonal polynomials on $[0,1]$ for the Jacobi weights $x^\alpha_i (1-x)^\beta$, with $\alpha_1,\ldots,\alpha_r,\beta >-1$
and $\alpha_i-\alpha_j \notin \mathbb{Z}$. They satisfy
\[   \int_0^1 P_{\vec{n}}(x) x^k x^{\alpha_j} (1-x)^\beta \, dx = 0, \qquad k=0,1,\ldots,n_j-1, \ 1 \leq j \leq r. \]
They are given by the Rodrigues formula
\begin{multline}  \label{JPRod}
   (-1)^{|\vec{n}|} \prod_{j=1}^r (|\vec{n}|+\alpha_j+\beta+1)_{n_j} \ (1-x)^\beta P_{\vec{n}}(x) \\
   =  \prod_{j=1}^r \left( x^{-\alpha_j} \frac{d^{n_j}}{dx^{n_j}} x^{n_j + \alpha_j} \right) (1-x)^{|\vec{n}|+\beta},
\end{multline}   
where the product of differential operators is the same as for the multiple Laguerre polynomials of the first kind.
One has
\[   \int_0^1 x^\gamma P_{\vec{n}}(x) (1-x)^\beta \, dx 
= (-1)^{|\vec{n}|} \frac{\prod_{i=1}^r (\alpha_i-\gamma)_{n_i}}{\prod_{i=1}^r (|\vec{n}|+\alpha_i+\beta+1)_{n_i}}
   \frac{\Gamma(\gamma+1)\Gamma(|\vec{n}|+\beta+1)}{\Gamma(|\vec{n}|+\beta+\gamma+2)}
\]
(see \cite[\S 23/3.2]{Ismail} which has an extra $(-1)^{|\vec{n}|}$ that shouldn't be there) so that
\[    \int_0^1 x^{n_j+\alpha_j}  P_{\vec{n}}(x) (1-x)^\beta \, dx 
= \frac{\prod_{i=1}^r \binom{n_j+\alpha_j-\alpha_i}{n_i} n_i!}{\prod_{i=1}^r (|\vec{n}|+\alpha_i+\beta+1)_{n_i}}
   \frac{\Gamma(n_j+\alpha_j+1)\Gamma(|\vec{n}|+\beta+1)}{\Gamma(|\vec{n}|+ n_j+\alpha_j+\beta+2)}.
\]
If we use (\ref{eq:1.11}) we then find for $1 \leq j \leq r$
\begin{multline}   \label{anjJP}
   a_{\vec{n},j} = \prod_{i=1,i\neq j}^r \frac{n_j+\alpha_j-\alpha_i}{n_j-n_i+\alpha_j-\alpha_i}
                   \prod_{i=1}^r \frac{|\vec{n}|+\alpha_i+\beta}{|\vec{n}|+n_i+\alpha_i+\beta} \\
 \times   \frac{n_j(n_j+\alpha_j)(|\vec{n}|+\beta)}{(|\vec{n}|+n_j+\alpha_j+\beta+1)(|\vec{n}|+n_j+\alpha_j+\beta)(|\vec{n}|+n_j+\alpha_j+\beta-1)}.
\end{multline}
We can write this as
\begin{multline*}
    a_{\vec{n},j} =  \frac{q_r(-|\vec{n}|-\beta)}{Q_{r,\vec{n}}(-|\vec{n}|-\beta)} 
     \frac{q_r(n_j+\alpha_j)}{Q'_{r,\vec{n}}(n_j+\alpha_j)} \\
  \times  \frac{(n_j+\alpha_j)(|\vec{n}|+\beta)}{(|\vec{n}|+n_j+\alpha_j+\beta+1)(|\vec{n}|+n_j+\alpha_j+\beta)(|\vec{n}|+n_j+\alpha_j+\beta-1)}, 
\end{multline*}
where the polynomials $q_r$ and $Q_{r,\vec{n}}$ are given, as before, by
\begin{equation}  \label{qQ}
    q_r(x) = \prod_{j=1}^r (x-\alpha_j), \qquad Q_{r,\vec{n}}(x) = \prod_{j=1}^r (x-n_j-\alpha_j).  
\end{equation}
Observe that for $r=1$ we retrieve the recurrence coefficients 
\[ a_n^2 = \frac{n(n+\alpha)(n+\beta)(n+\alpha+\beta)}{(2n+\alpha+\beta-1)(2n+\alpha+\beta)^2(2n+\alpha+\beta+1)}  \]
for monic Jacobi polynomials on $[0,1]$. 

The other recurrence coefficients can be obtained by $b_{\vec{n},k} = \delta_{\vec{n}} - \delta_{\vec{n}+\vec{e}_k}$, where
$\delta_{\vec{n}}$ is the coefficient of $x^{|\vec{n}|-1}$ of the polynomial $P_{\vec{n}}(x)$:
\[   P_{\vec{n}}(x) = x^{|\vec{n}|} + \delta_{\vec{n}} x^{|\vec{n}|-1} + \cdots.  \]
For Jacobi-Pi\~neiro polynomials we have
\begin{equation}  \label{JPdelta}
   \delta_{\vec{n}} = - (|\vec{n}|+\beta) \frac{q_r(-|\vec{n}|-\beta)}{Q_{r,\vec{n}}(-|\vec{n}|-\beta)} + \beta ,
\end{equation}
where the polynomials $q_r$ and $Q_{r,\vec{n}}$ are given in (\ref{qQ}).
To see that this is true, one can use the Rodrigues formula (\ref{JPRod}) to find that
\[    \left( x^{\alpha_i+1} (1-x)^{\beta+1} P_{\vec{n}-\vec{e}_i}^{(\vec{\alpha}+\vec{e}_i,\beta+1)}(x) \right)'
     = -(|\vec{n}|+\alpha_i+\beta+1) x^{\alpha_i} (1-x)^{\beta} P_{\vec{n}}^{(\vec{\alpha},\beta)}(x)  \]
(see also \cite[Eq. (1.4)]{JCWVA}). If we identify the coefficients in this identity, then
\[  (|\vec{n}|+\alpha_i+\beta+1) \delta_{\vec{n}}(\vec{\alpha},\beta) 
   = (|\vec{n}|+\alpha_i+\beta+) \delta_{\vec{n}-\vec{e}_i}(\vec{\alpha}+\vec{e}_i,\beta+1) - |\vec{n}|-\alpha_i. \]
This can be used to prove (\ref{JPdelta}) by induction on $r$. Observe that we can use a decomposition in partial fractions
to write
\[    \frac{q_r(z)}{Q_{r,\vec{n}}(z)} = 1 + \sum_{j=1}^r \frac{q_r(n_j+\alpha_j)/Q'_{r,\vec{n}}(n_j+\alpha_j)}{z-n_j-\alpha_j}, \]
so that
\[  \frac{q_r(-|\vec{n}|-\beta)}{Q_{r,\vec{n}}(-|\vec{n}|-\beta)} 
  = 1 - \sum_{j=1}^r \frac{q_r(n_j+\alpha_j)/Q'_{r,\vec{n}}(n_j+\alpha_j)}{|\vec{n}|+n_j+\alpha_j+\beta}  \]
and
\[   \delta_{\vec{n}} = - |\vec{n}| +  
 (|\vec{n}|+\beta) \sum_{j=1}^r \frac{q_r(n_j+\alpha_j)/Q'_{r,\vec{n}}(n_j+\alpha_j)}{|\vec{n}|+n_j+\alpha_j+\beta}. \]
Furthermore
\[    \sum_{j=1}^r q_r(n_j+\alpha_j)/Q'_{r,\vec{n}}(n_j+\alpha_j) = |\vec{n}| \]
so that a little calculus gives
\begin{equation*}  \label{JPdelta2}
     \delta_{\vec{n}} = - \sum_{j=1}^r \frac{(n_j+\alpha_j)q_r(n_j+\alpha_j)/Q'_{r,\vec{n}}(n_j+\alpha_j)}{|\vec{n}|+n_j+\alpha_j+\beta}. 
\end{equation*}
Note that for $r=1$ this indeed gives
\[   \delta_n = \frac{-n(n+\alpha)}{2n+\alpha+\beta}, \]
from which
\[    b_n = \delta_n - \delta_{n+1} = \frac12 + \frac{\beta^2-\alpha^2}{2(2n+\alpha+\beta)(2n+\alpha+\beta+2)}, \]
which is indeed the recurrence coefficient $b_n$ for the Jacobi polynomials on $[0,1]$. 

We can find a nice closed form for the sum $\sum_{j=1}^r a_{\vec{n},j}$. We have
\begin{multline*} \sum_{j=1}^r a_{\vec{n},j} = (|\vec{n}|+\beta) \frac{q_r(-|\vec{n}|-\beta)}{Q_{r,\vec{n}}(-|\vec{n}|-\beta)} \\
 \times   \sum_{j=1}^r \frac{q_r(n_j+\alpha_j)}{Q_{r,\vec{n}}'(n_j+\alpha_j)} \frac{n_j+\alpha_j}
  {(|\vec{n}|+n_j+\alpha_j+\beta+1)(|\vec{n}|+n_j+\alpha_j+\beta)(|\vec{n}|+n_j+\alpha_j+\beta-1)} ,
\end{multline*}
and for the sum in this expression, we can use the residue theorem to find
\begin{multline*}
  \sum_{j=1}^r \frac{q_r(n_j+\alpha_j)}{Q_{r,\vec{n}}'(n_j+\alpha_j)} \frac{n_j+\alpha_j}
  {(|\vec{n}|+n_j+\alpha_j+\beta+1)(|\vec{n}|+n_j+\alpha_j+\beta)(|\vec{n}|+n_j+\alpha_j+\beta-1)} \\
   = \frac{1}{2\pi i} \int_\Gamma \frac{q_r(z)}{Q_{r,\vec{n}}(z)} \frac{z\, dz}{(z+|\vec{n}|+\beta+1)(z+|\vec{n}|+\beta)(z+|\vec{n}|+\beta-1)},
\end{multline*}
where $\Gamma$ is a closed contour around all the zeros $n_i+\alpha_i$ $(i=1,\ldots,r)$ of $Q_{r,\vec{n}}$ but with
$-|\vec{n}|-\beta$, $-|\vec{n}|-\beta \pm 1$ outside the contour. If we choose $\Gamma$ so that $0$ is inside and 
change the variable $z =1/\xi$, then the integral becomes
\[  \frac{1}{2\pi i} \int_{\Gamma^*} \frac{q_r^*(\xi)}{Q_{r,\vec{n}}^*(\xi)} 
   \frac{d\xi}{(1+\xi(|\vec{n}|+\beta+1))(1+\xi(|\vec{n}|+\beta))(1+\xi(|\vec{n}|+\beta-1))},  \]
where $\Gamma^*$ is now a contour around $0$ with the poles $-1/(|\vec{n}+\beta)$ and $-1/(|\vec{n}|+\beta\pm 1)$ inside and the
zeros of $Q_{r,\vec{n}}^*$ outside. The polynomials $q_r^*$ and $Q_{r,\vec{n}}^*$ are the reversed polynomials (as in \S \ref{sec:ML1}). 
The residue theorem therefore evaluates this contour integral as
\[     \frac12 \frac{q_r(-|\vec{n}|-\beta-1)}{Q_{r,\vec{n}}(-|\vec{n}|-\beta-1)} (|\vec{n}|+\beta-1)
       -  \frac{q_r(-|\vec{n}|-\beta)}{Q_{r,\vec{n}}(-|\vec{n}|-\beta)} (|\vec{n}|+\beta)
       + \frac12 \frac{q_r(-|\vec{n}|-\beta+1)}{Q_{r,\vec{n}}(-|\vec{n}|-\beta+1)} (|\vec{n}|+\beta+1), \]
and we conclude that
\begin{multline*}
  \sum_{j=1}^r a_{\vec{n},j} =  (|\vec{n}|+\beta) \frac{q_r(-|\vec{n}|-\beta)}{Q_{r,\vec{n}}(-|\vec{n}|-\beta)} 
   \left( \frac12 \frac{q_r(-|\vec{n}|-\beta-1)}{Q_{r,\vec{n}}(-|\vec{n}|-\beta-1)} (|\vec{n}|+\beta-1) \right. \\
      \left.  -\  \frac{q_r(-|\vec{n}|-\beta)}{Q_{r,\vec{n}}(-|\vec{n}|-\beta)} (|\vec{n}|+\beta)
       + \frac12 \frac{q_r(-|\vec{n}|-\beta+1)}{Q_{r,\vec{n}}(-|\vec{n}|-\beta+1)} (|\vec{n}|+\beta+1) \right).  
\end{multline*}
Observe that the expression between brackets is $-1/2$ times the second difference $\Delta \nabla$ of $zq_r(z)/Q_{r,\vec{n}}(z)$
evaluated at $z=-|\vec{n}|-\beta$ and that this expression vanishes when $|\vec{n}|=0$.

\begin{verbatim}
Walter Van Assche
Department of Mathematics
Katholieke Universiteit Leuven
Celestijnenlaan 200 B box 2400
BE-3001 Leuven
Belgium
walter@wis.kuleuven.be
\end{verbatim}

\end{document}